\documentclass[]{amsart}
\usepackage{amsmath,amsfonts,amssymb,amsthm}
\usepackage{enumerate, color}
\usepackage{mathrsfs}
 \DeclareMathOperator{\ord}{ord}
\DeclareMathOperator{\Diff}{Diff} \DeclareMathOperator{\re}{Re}

\DeclareMathOperator{\sing}{Sing}
\DeclareMathOperator{\sat}{Sat}

\def\N{\mathbb N}
\def\R{\mathbb R}
\def\C{\mathbb C}
\def\Q{\mathbb Q}
\def\Z{\mathbb Z}

\newcommand\parcial[1]{\dfrac{\partial}{\partial #1}}
\def\N{\mathbb N}
\def\R{\mathbb R}
\def\C{\mathbb C}
\def\Q{\mathbb Q}
\def\Z{\mathbb Z}

\theoremstyle{plain}
\newtheorem{theorem}{Theorem}[section]
\newtheorem{theoremintro}{Theorem}
\newtheorem{proposition}[theorem]{Proposition}
\newtheorem{lemma}[theorem]{Lemma}
\theoremstyle{definition}
\newtheorem*{remarkintro}{Remark}

\newtheorem{remark}[theorem]{Remark}
\newcommand{\obra}[3]{{\sc #1} {\em #2}. {#3}.}

\author{Lorena L\'opez-Hernanz}
\address{Lorena L\'opez-Hernanz\\ Departamento de F\'isica y Matem\'aticas, Universidad de Alcal\'a}
\email{lorena.lopezh@uah.es}
\author{Rudy Rosas}
\address{Rudy Rosas\\Departamento de Ciencias, Pontificia Universidad Cat\'olica del Per\'u}
\email{rudy.rosas@pucp.pe}
\date{}
\thanks{First author partially supported by Ministerio de Ciencia e Innovación, Spain, PID2019-105621GB-I00; second author, by Vicerrectorado de Investigaci\'on de la Pontificia Universidad Cat\'olica del Per\'u}
\title{A flower theorem in dimension two}
\begin{document}
\maketitle
\begin{abstract}
We prove a two-dimensional analog of Leau-Fatou flower theorem for non-degenerate reduced tangent to the identity biholomorphisms. 
\end{abstract}

\section{Introduction}

Let $F\in\Diff(\C^n,0)$ be a germ of biholomorphism tangent to the identity. In dimension one, the dynamics of $F$ is completely described by Leau-Fatou flower theorem \cite{Lea, Fat}, which guarantees the existence of simply connected domains with 0 in their boundary, covering a punctured neighborhood of the origin, which are either stable for $F$ or for $F^{-1}$; moreover, in each of these domains $F$ is conjugated to the unit translation. 

In dimension two, no complete description of the dynamics of $F$ is known. Some partial analogs of Leau-Fatou flower theorem have been obtained, guaranteeing either the existence of one-dimensional stable manifolds \cite{Eca, Hak, Aba, LR} or of two-dimensional ones \cite{Hak2, Viv}. With no extra assumptions on $F$, the most general result is due to Abate \cite{Aba}, who showed that $F$ always supports some stable dynamics: either $F$ has a curve of fixed points or there exist one-dimensional stable manifolds of $F$ with 0 in their boundary. 

The proof of the mentioned results is crucially based on a resolution theorem for $F$, which reduces the study of the dynamics of $F$ to some combinatorial data of the resolution and the study of the local dynamics of some reduced models of the transform of $F$. This resolution theorem was introduced by Abate in \cite{Aba} and is based on the corresponding result for vector fields due to Seidenberg \cite{Sei}. 
The resolution theorem, as we explain in the appendix, guarantees the existence of a finite composition of blow-ups $\pi:(M,E)\to(\C^2,0)$, with $E=\pi^{-1}(0)$, which transforms $F$ into a map $\tilde F:(M,E)\to(M,E)$ that fixes $E$ pointwise such that for every $p\in E$ the germ of $\tilde F$ at $p$ is analytically conjugate to one of the following reduced models:
\begin{enumerate}[(i)]
\item $\tilde{F}(x,y)=\left(x+x^{M}y^N[1+A(x,y)],y+ x^{M}y^{N}B(x,y)\right)$,
where $M,N\in\mathbb{Z}_{\ge 0}$,  $(M,N)\notin\{ (0,0), (1,0)\}$, $\ord A\ge1$ and $B\in(y)$ if $N\ge 1$.

\item\label{nondegeneratediffeo} $\tilde{F}(x,y)=\left(x+x^{M+1}y^N\left[a+A(x,y)\right], y+x^My^{N}\left[by+B(x,y)\right]\right)$,
where $M\ge 1$, $N\ge 0$, $ab\neq 0$, ${a}/{b}\notin\mathbb{Q}_{>0}$,  $\ord A\ge 1$, $\ord B\ge 2$ and $B\in (y)$ if $N\ge 1$. 

\item $\tilde{F}(x,y)=\left(x+x^{M}y^N\left[x+A(x,y)\right], y+x^My^{N}B(x,y)\right)$,
where ${M, N\ge 0}$, $M+N\ge 1$,  $\ord A, \ord B\ge 2$, $A\in (x)$ if $M\ge 1$, $B\in(y)$ if $N\ge 1$ and $x+A$ and $B$ have no common factors.

\end{enumerate}

Models (ii) and (iii) correspond to the points $p\in E$ which are singular, in Abate's terminology \cite{Aba} (i.e. the points which are strictly singular for the associated vector field), and model (i) corresponds to the points $p\in E$ which are not singular. The dynamics of biholomorphisms of the form (i) is described in \cite[Proposition 2.1]{Aba} when $M=0$ and in \cite[Theorem 5.3]{Br1} when $N=0$.

In this paper we study the dynamics of biholomorphisms of the form (ii), which are called reduced non-degenerate in analogy with the standard terminology for vector fields. They are the only models that appear at singular points in the resolution of a generic biholomorphism.

Actually, we consider in our study a slightly more general class of biholomorphisms, since we do not impose the non-resonance condition $a/b\not\in\Q_{>0}$. We distinguish two cases, according to the fixed point set of $F$. As a first case, we consider biholomorphisms of the form
\begin{align}\label{eq:noncornerdiffeo}
F(x,y)=\left(x+x^{M+1}\left[a+A(x,y)\right], y+x^M\left[cx+by+B(x,y)\right]\right),
\end{align}
where $M\ge1$, $a,b,c\in\mathbb{C}$, $ab\neq0$,  $\ord A\ge1$ and $\ord B\ge 2$. This type of biholomorphisms appear, for instance, after one blow-up centered at a \emph{non-degenerate characteristic direction} (see \cite{Hak}). Écalle \cite{Eca} and Hakim \cite{Hak} showed that in this case there exist one-dimensional stable manifolds for $F$ with 0 in their boundary, called parabolic curves. Moreover, if $a$ and $b$ satisfy the condition
\begin{align}\label{eq:hipononcorner}
\re\left(b/a\right)>0,
\end{align}  
Hakim proved in \cite{Hak2} (see also \cite{Ari-R}) the existence of two-dimensional stable manifolds with 0 in their boundary, called parabolic domains, where $F$ is analytically conjugate to the map $(z,w)\mapsto(z+1,w)$.

As a second case, we consider biholomorphisms of the form
\begin{align}\label{eq:cornerdiffeo}
F(x,y)=\left(x+x^{M+1}y^N\left[a+A(x,y)\right], y+x^My^{N+1}\left[b+B(x,y)\right]\right),
\end{align}
where $M, N\ge1$,  $ab\neq 0$ and $\ord A,\ord B\ge 1$. If $a$ and $b$ are such that $aM+bN\neq0$ and satisfy the condition
\begin{align}\label{eq:hipo}\re\left(\frac a{aM+bN}\right)>0\quad\text{ and }\quad \re\left(\frac b{aM+bN}\right)>0,
\end{align} 
Vivas proved in \cite{Viv} the existence of parabolic domains for $F$.

The main result of this paper is an analog of Leau-Fatou flower theorem for this type of biholomorphisms.
\begin{theoremintro}\label{FatouFlower}
Let $F$ be a biholomorphism of the form \eqref{eq:noncornerdiffeo} or \eqref{eq:cornerdiffeo}. In the first case, assume that $F$ satisfies condition \eqref{eq:hipononcorner} and set $d=M$ and $N=0$; in the second case, assume that $F$ satisfies condition \eqref{eq:hipo} and set $d=(M,N)$. Then, in any neighborhood of the origin there exist $d$ pairwise disjoint connected open sets $\Omega^+_0,\Omega^+_1,\dots,\Omega^+_{d-1}$, with $0\in\partial\Omega^+_k$ for all $k$, and $d$ pairwise disjoint connected open sets  $\Omega^-_0,\Omega^-_1,\dots,\Omega^-_{d-1}$, with $0\in\partial\Omega^-_k$ for all $k$, such that the following assertions hold:  
\begin{enumerate}[1)]
\item\label{item:parabolicdomains} The sets $\Omega^+_k$ are invariant for $F$ and $F^j\to 0$ as $j\to +\infty$ compactly on $\Omega^+_k$ for all $k$, and the sets $\Omega^-_k$ are invariant for $F^{-1}$ and $F^{-j}\to 0$ as $j\to +\infty$ compactly on $\Omega^-_k$ for all $k$.
\item\label{item:coverneigh} The sets  $\Omega^+_0,\dots,\Omega^+_{d-1},\Omega^-_0,\dots,\Omega^-_{d-1}$ together with the fixed set $\{xy^N=0\}$ cover a neighborhood of the origin. 
\item \label{tercero+} For each $k$, there exist biholomorphisms $\varphi^+_k: \Omega^+_k\to W_k^+\subset\C^2$ and $\varphi^-_k: \Omega^-_k \to W_k^-\subset\C^2$, with $W_k^+,W_k^-\subset\C\times\C^*$ if $N\ge1$, with the following properties:
 \begin{enumerate} [a)] 
 \item\label{tercero1+} $\varphi^+_k$ and $\varphi^-_k$  conjugate $F$ with the map $(z,w)\mapsto (z+1,w)$.
 \item\label{tercero1.5+} The sets $W_k^+$ and $W_k^-$ satisfy
  $$\bigcup_{\pm j\in\N} [W_k^{\pm}-(j,0)]=\C^2 \ \text{ if } N=0; \ \quad \bigcup_{\pm j\in\N} [W_k^{\pm}-(j,0)]=\C\times\C^* \ \text{ if } N\ge1.$$
\end{enumerate}  
\end{enumerate}
\end{theoremintro}

Our second result shows that if conditions \eqref{eq:hipononcorner} or \eqref{eq:hipo} are strictly not satisfied, then $F$ has generic finite orbits in some neighborhood of the origin and so no two-dimensional stable sets. 

\begin{theoremintro}\label{saddlebehavior} 
Let $F$ be a biholomorphism of the form \eqref{eq:noncornerdiffeo} or \eqref{eq:cornerdiffeo}. In the first case, assume that
$\re(b/a)<0$; in the second case, assume that either
$$\re\left(\frac{a}{aM+bN}\right)<0\quad\text{ or }\quad \re\left(\frac{b}{aM+bN}\right)<0.$$ 
Then there exists a neighborhood $\mathcal{U}$ of the origin and there exist sets 
$\mathcal{P}^+,\mathcal{P}^-\subset\mathcal{U}$, which are one-dimensional submanifolds of $\mathcal U$ if $F$ is of the form \eqref{eq:noncornerdiffeo} and are empty otherwise, such that the following properties hold:
given $p\in \mathcal{U}\setminus \mathcal P^+$ outside the fixed set, there exist $j\in\mathbb{N}$ such that $F^j(p)\notin\mathcal{U}$; given $p\in \mathcal{U}\setminus\mathcal P^-$ outside the fixed set, there exist $j\in\mathbb{N}$ such that $F^{-j}(p)\notin\mathcal{U}$. If $F$ is of the form \eqref{eq:noncornerdiffeo}, $\mathcal P^+$ is the set of points in $\mathcal U$ that are attracted under $F$ to the parabolic curves of $F$, and $\mathcal P^-$ is the set of points in $\mathcal U$ that are attracted under $F^{-1}$ to the parabolic curves of $F^{-1}$.
\end{theoremintro}

\begin{remarkintro} \label{illustrating} 
To show the necessity of the hypotheses in Theorems \ref{FatouFlower} and \ref{saddlebehavior}, consider, for $M\ge1$, $N\ge0$ and $a,b\in\mathbb{C}^*$ that do not satisfy those hypotheses, the biholomorphism given by the time-1 flow of the vector field 
$X=x^My^N\left[ax{\partial_ x}+by{\partial_ y}\right]$
and let us show that for any neighborhood $\mathcal{U}$ of the origin there exists $p\in\mathcal{U}$ outside the fixed set such that
the orbit $\{F^j(p)\colon j\in\mathbb{Z}\}$ is contained in $\mathcal{U}$ and bounded away from the origin. If $(x(t),y(t))$ is a solution of $X$ and we set $P(t)=x(t)^My(t)^N$, we have 
$P'=(aM+bN)P^2$,  $x'=aPx$ and $y'=bPy$.
Suppose first that $aM+bN=0$. 
Then $P(t)=P(0)$, $x(t)=x(0)e^{aP(0)t}$ and $y(t)=y(0)e^{bP(0)t}$, so
$$(x_j,y_j)=F^j(x,y)=\left(xe^{ax^My^Nj},ye^{bx^My^Nj}\right),\quad j\in\mathbb{Z}.$$
Note that, since $a/b=-N/M\in\mathbb R$, in any neighborhood $\mathcal{U}$ of the origin we can take $(x,y)$ arbitrarily small with $xy\neq0$ such that $\re \left(ax^{M}y^N\right)=\re \left(bx^{M}y^N\right)=0$. In this case, the expression 
above shows that $|x_j|=|x|$, $|y_j|=|y|$ for all $j\in\mathbb{Z}$, so $\{(x_j,y_j):j\in\Z\}$ is bounded away from the origin and contained in $\mathcal{U}$ provided $(x,y)$ is small enough. Suppose now that $aM+bN\neq 0$, $\re\left(\frac{a}{aM+bN}\right)\ge0$ and $\re\left(\frac{b}{aM+bN}\right)=0$. If we denote $x(0)=x$ and $y(0)=y$, we get by integration that $P(t)=x^My^N(1-(aM+bN)x^My^Nt)^{-1}$ and
\begin{align*}
x(t)&=x\left[1-(aM+bN)x^My^Nt\right]^{-a/(aM+bN)}\\
y(t)&=y\left[1-(aM+bN)x^My^Nt\right]^{-b/(aM+bN)},
\end{align*}
defined for all $t\in\mathbb{R}$ provided $(aM+bN){x^My^N}\notin \mathbb{R}$. If we set $\tfrac{a}{aM+bN}=\alpha_1+i\alpha_2$, with $\alpha_1\ge0$, and $\frac{b}{aM+bN}=i\beta$, we have
$$|x(t)|=|x||1-(aM+bN)x^My^Nt|^{-\alpha_1}e^{\alpha_2\arg\left(1-(aM+bN)x^My^Nt\right)}$$
and 
$$|y(t)|=|y|e^{\beta\arg\left(1-(aM+bN)x^My^Nt\right)},$$
so $$|x(t)|\le C e^{|\alpha_2|\pi}|x|;\quad e^{-|\beta|\pi}|y|\le |y(t)|\le e^{|\beta|\pi}|y|$$
for some $C>0$ and for all $t\in\R$. Therefore, given a neighborhood $\mathcal{U}$ of the origin, if we choose $(x,y)$ as above 
and sufficiently small then its orbit is contained in $\mathcal{U}$ and bounded away from the origin.
\end{remarkintro}

\section{Existence of parabolic domains and invariant functions}\label{constructing}
Sections \ref{constructing} to \ref{secfatou} are devoted to the proof of Theorem~\ref{FatouFlower}. In this section we show, for $F$ satisfying the hypotheses of Theorem~\ref{FatouFlower}, the existence of parabolic domains for $F$ and the existence of an $F$-invariant function in each of these domains. 

We will use the following statement of Leau-Fatou flower theorem (see, for instance, \cite[Théorème 2.3.1]{Lor}). Given $d\ge1$, we denote, for $\varepsilon>0$ and $\theta\in(0,\pi/2)$, 
$$S(\varepsilon,\theta)=\left\{z\in\C:|z^d|<\varepsilon, |\arg(z^d)|<\theta\right\},$$
which is a union of $d$ sectors of opening $2\theta/d$ bisected by the half-lines $e^{2\pi i k/d}\R^+$, for $k\in\{0,\dots,d-1\}$. We also define the set
$$\widetilde S(\varepsilon,\theta)=S(\varepsilon,\theta)\cup \left\{z\in\C: \left|z^d-\tfrac{\varepsilon}2 e^{-i\theta}\right|<\tfrac{\varepsilon}2\right\}\cup \left\{z\in\C:\left|z^d-\tfrac{\varepsilon}2 e^{i\theta}\right|<\tfrac{\varepsilon}2\right\},$$
which is a union of $d$ sectorial domains of opening $(\pi+2\theta)/d$, also bisected by the half-lines $\{z\in\C:z^d\in\R^+\}$. 

\begin{theorem}[Leau-Fatou]\label{the:flowerdim1}
Let $f:(\C,0)\to(\C,0)$ be a biholomorphism of the form $f(z)=z-\frac1 dz^{d+1}+O(z^{d+2})$, $d\ge1$. For any $\theta\in(0,\pi/2)$ there exist constants $\varepsilon_0=\varepsilon_0(\theta)>0$, $c=c(\varepsilon_0, \theta)>1$ and $C=C(\varepsilon_0, \theta)>0$ such that for every $\varepsilon\le \varepsilon_0$ and every component $\widetilde S$ of $\widetilde S(\varepsilon,\theta)$ we have that $f(\widetilde S)\subset \widetilde S$ and
$$\lim_{j\to\infty}j(f^j(z))^d=1 \quad \text{ and }\quad |f^j(z)|^d\le c\dfrac{|z|^d}{1+j|z|^d}$$
for every $z\in \widetilde S$ and $j\in\N$. Moreover, if $S$ is the component of $S(\varepsilon,\theta)$ contained in $\widetilde S$, it also holds that 
$f(S)\subset S$ and $f^j(z)\in S$ for every $z\in \widetilde S$ and every $j\ge C/|z|^d$.
\end{theorem}

\begin{remark}\label{rk:repelling}
An analogous result follows for $f^{-1}(z)=z+\frac1d z^{d+1}+O(z^{d+2})$ and the sets
$$S^-(\varepsilon,\theta)=\left\{z\in\C:|z^d|<\varepsilon, |\arg(z^d)-\pi|<\theta\right\}$$
and 
$$\widetilde S^-(\varepsilon,\theta)=S^-(\varepsilon,\theta)\cup \left\{z\in\C: \left|z^d-\tfrac{\varepsilon}2 e^{i(\pi-\theta)}\right|<\tfrac{\varepsilon}2\right\}\cup \left\{z\in\C:\left|z^d-\tfrac{\varepsilon}2 e^{i(\pi+\theta)}\right|<\tfrac{\varepsilon}2\right\}.$$
\end{remark}

\strut

Consider first a biholomorphism $F=(F_1,F_2)$ of the form \eqref{eq:noncornerdiffeo}. Up to a linear change of coordinates of the form $(x,y)\mapsto(\alpha x,y)$, we can assume that $a\in\R^-$. In this case, Hakim proved in \cite{Hak} that if $r$ is small enough then for any $k\in\{0,\dots,M-1\}$ there exists an injective holomorphic map $u_k:D_{r,k}\to\C$, with $|u_k(x)|\le K|x\log x|$ for some $K>0$ and for all $x\in D_{r,k}$, such that $u_k(F_1(x,u_k(x)))=F_2(x,u_k(x))$, where $D_{r,k}$ is the component of $\{x\in\C:|x^M-r|<r\}$ bisected by $e^{2\pi i k/M}\R^+$. Moreover, with the small modification of her proof introduced in \cite[Lemma 4.4]{Lop}, we can enlarge the domain of definition of $u_k$ to the connected component of $\widetilde S(\varepsilon,\theta)$ bisected by $e^{2\pi i k/M}\R^+$, for any $\theta\in(0,\pi/2)$ and for $\varepsilon$ small enough. We denote this connected component $\widetilde S_k(\varepsilon,\theta)$. Then, making the sectorial change of coordinates 
\begin{align}\label{sectorialchange}(x,y)\in \widetilde S_k(\varepsilon,\theta)\times \C\mapsto (x,y-u_k(x)),
\end{align}
we can write
$$F(x,y)=\left(x+x^{M+1}\left[a+O_1(x,y)\right], y+x^My\left[b+O_1(x,y)\right]\right),$$
where we use the notation $O_1(x,y)=O(x,x\log x,y)$. Hence, in sectorial coordinates we can write $F$ in the form \eqref{eq:cornerdiffeo} with $N=0$ and $O_1$ instead of $O$, so the proof of Theorem~\ref{FatouFlower} in this case will be essentially derived from case \eqref{eq:cornerdiffeo} for $N=0$ (observe that condition \eqref{eq:hipononcorner} is exactly condition \eqref{eq:hipo} with $N=0$). Thus, in order to prove Theorem~\ref{FatouFlower} we consider from now on a map $F$ of the form \eqref{eq:cornerdiffeo} allowing $N\ge 0$ and satisfying condition \eqref{eq:hipononcorner}. In Section~\ref{secfatou} we provide further clarifications to complete the case of biholomorphisms of the form \eqref{eq:noncornerdiffeo}.

Set $d=M$ if $N=0$ and $d=\gcd(M,N)$ otherwise. Applying a linear change of coordinates of the form
$(x,y)\mapsto (\alpha x,\beta y)$, we obtain the same expression for $F$ but with $a$ and $b$ respectively
replaced by 
$\tilde a=-\frac{a}{aM+bN}$ and $ \tilde b=-\frac{b}{aM+bN},$
so hypothesis \eqref{eq:hipo} becomes $\re \tilde{a}<0$ and $\re \tilde{b}<0$, and we have $\tilde{a}M+\tilde{b}N=-1$. Thus, we directly assume that 
$$\re a<0, \ \re b<0 \ \text{ and } \ aM+bN=-1.$$  

Set $m=\frac{M}{d}$ and $n=\frac{N}{d}$. Given a point $(x,y)$ in the domain of definition of $F$, we denote $(x_j,y_j)=F^j(x,y)$, for all $j\ge 0$. Observe that
\begin{equation}\label{eq:x^my^n}
x_1^my_1^n=x^my^n-\tfrac 1d(x^my^n)^{d+1}+(x^my^n)^{d+1}O(x,y),
\end{equation} 
so if $x_j,y_j$ are small we can argue as in Theorem~\ref{the:flowerdim1} and we find $d$ attracting directions for the variable $z=x^my^n$, given by the half-lines $e^{2\pi i k/d}\R^+$, $0\le k\le d-1$.

Fix $\gamma>0$ such that $\re a+\frac\gamma d<0$ and $\re b+\frac\gamma d<0$, and denote $\epsilon=\min\{1,n\}$. For any $\theta\in\left(0, \frac{\pi}2\right)$ and $\varepsilon, \delta>0$ we consider, for any $k\in\{0,\dots,d-1\}$, the sets $D_k=D_k(\varepsilon,\theta,\delta)$ and $U_k=U_k(\varepsilon,\theta)$ defined by
$$D_k=\left\{(x,y)\in\C^2:|x^my^n|<\varepsilon^{1/d}, \left|\arg(x^my^n)-\tfrac{2\pi k}d\right|<\tfrac \theta d, |\epsilon x|<\delta, |y|<\delta\right\}$$
and 
$$U_k=\left\{(x,y)\in\C^2:|x^my^n|<\varepsilon^{1/d}, \left|\arg(x^my^n)-\tfrac{2\pi k}d\right|<\tfrac\theta d, |\epsilon x|\le|x^my^n|^\gamma, |y|\le|x^my^n|^\gamma\right\}.$$

\begin{proposition}\label{parabola} 
If $\varepsilon$ and $\theta$ are sufficiently small then $F(U_k)\subset U_k$ and $F^j\to 0$ as $j\to \infty$ uniformly on $U_k$, for each $k$. Moreover, any orbit of $F$ that converges to 0 eventually lies in $U_k$ for some $k$.
\end{proposition}
\begin{proof} We denote, as above, $(x_j,y_j)=F^j(x,y)$. To show the first two properties of $U_k$ we assume, without loss of generality, that $k=0$.
By equation~\eqref{eq:x^my^n}, arguing as in Leau-Fatou flower theorem, for any $\theta\in\left(0,\frac{\pi}2\right)$ we find
$\delta_\theta$ depending only on $\theta$ such that for $\varepsilon$ small enough we have that if $(x,y)\in D_0(\varepsilon,\theta, \delta_\theta)$ then
$$|x_1^my_1^n|<\varepsilon^{1/d} \ \ \text{and} \ \ \left|\arg(x_1^my_1^n)\right|<\theta/d.$$
On the other hand, if $x,y\in\C^*$ are small enough, we have
\begin{align*}
\dfrac{|y_1|}{|x_1^my_1^n|^\gamma}&=\dfrac{|y|\left|1+b(x^my^n)^d(1+O(x,y))\right|}{|x^my^n|^\gamma\left|1-\frac{\gamma}{d} (x^my^n)^d(1+O(x,y))\right|}\\
&=\dfrac{|y|}{|x^my^n|^\gamma}\left|1+(x^my^n)^d\left[b+\tfrac{\gamma}{d}+O(x,y)\right]\right|,
\end{align*}
and if $n\ge1$, analogously,
\begin{align*}
\dfrac{|x_1|}{|x_1^my_1^n|^\gamma}&=\dfrac{|x|}{|x^my^n|^\gamma}\left|1+ (x^my^n)^d\left[a+\tfrac{\gamma}{d}+O(x,y)\right]\right|.
\end{align*}
Hence, since $\re a+\frac\gamma d<0$ and $\re b+\frac\gamma d<0$, there exist $\delta_0>0$ and $\eta>0$ such that, 
if $\varepsilon$ and $\theta$ are small enough, then
\begin{equation}\label{eq:boundeta}
\dfrac{|y_1|}{|x_1^my_1^n|^\gamma}\le
\dfrac{|y|}{|x^my^n|^\gamma}\left(1-\eta\left|x^my^n\right|^d\right)\le \dfrac{|y|}{|x^my^n|^\gamma}
\end{equation}
and, if $n\ge1$,
$$ \dfrac{|x_1|}{|x_1^my_1^n|^\gamma}\le \dfrac{|x|}{|x^my^n|^\gamma}\left(1-\eta\left|x^my^n\right|^d\right)\le 
\dfrac{|x|}{|x^my^n|^\gamma}$$
for any $(x,y)\in D_0(\varepsilon,\theta,\delta_0)$. Since $U_0(\varepsilon,\theta)\subset D_0(\varepsilon,\theta,\delta)$ for all $\varepsilon>0$ such that 
$\varepsilon^{\gamma/d}<\delta$, we conclude that if $\varepsilon$ and $\theta$ are sufficiently small we have that $(x_1,y_1)\in U_0$ for any $(x,y)\in U_0$ and therefore $U_0$ is invariant by $F$.  
Let us show that the orbit of any point $(x,y)\in U_0$ tends uniformly to the origin. Arguing again as in Leau-Fatou theorem, there exists a constant $c>0$ such that
\begin{equation}\label{LFbounds0}
|x_j^my_j^n|^d\le c\frac{|x^my^n|^d}{1+j|x^my^n|^d}\le \frac c{j}
\end{equation}
for every $(x,y)\in U_0$ and every $j$. Since $|y_j|\le|x_j^my_j^n|^\gamma$ for all $j$, we have that $|y_j|\le (c/j)^{\gamma/d}$
for every $j$, which shows that $y_j\to 0$ uniformly on $U_0$. The uniform convergence to 0 of $x_j$ follows analogously in case $n\ge1$ and from Leau-Fatou theorem in case $n=0$.

Consider now an orbit $(x_j,y_j)$ converging to 0, and let us show that it eventually lies in $U_k=U_k(\varepsilon,\theta)$ for some $k$. Since $x_j^my_j^n\to0$ we have, as in Leau-Fatou theorem, that there exist $j_0\in\mathbb N$ and $k\in\{0,\dots, d-1\}$ such that $|x_j^my_j^n|<\varepsilon^{1/d}$ and $|\arg(x_j^my_j^n)-\frac{2\pi k}{d}|<\frac\theta d$ for all $j\ge j_0$, and moreover $\lim_ {j\to\infty} j(x_j^my_j^n)^d=1$. Therefore we have, up to increasing $j_0$ if necessary, that $(x_j,y_j)\in D_k(\varepsilon,\theta,\delta_0)$ for all $j\ge j_0$. Then, using inequality~\eqref{eq:boundeta}, we obtain
\begin{align*}
\dfrac{|y_j|}{|x_j^my_j^n|^\gamma}&\le
\dfrac{|y_{j_0}|}{|x_{j_0}^my_{j_0}^n|^\gamma}\prod\limits_{l=j_0}^{j-1}
\left(1-\eta\left|x_l^my_l^n\right|^d\right)
\end{align*} 
for all $j\ge j_0$. Since $\lim_ {j\to\infty} j(x_j^my_j^n)^d=1$, we have that $\sum_{l\ge0}|x_l^my_l^n|^d=+\infty$, so the product above tends to 0 and so does $|y_j|/|x_j^my_j^n|^\gamma$. Hence  $|y_j|\le|x_j^my_j^n|^\gamma$ for $j$ large enough, and if $n\ge1$ analogously 
$|x_j|\le |x_j^my_j^n|^\gamma$ for $j$ large enough, so $(x_j,y_j)\in U_k$.
\end{proof}

\begin{remark}\label{rk:largerdomains}
It also holds that if $\varepsilon$, $\theta$ and $\delta$ are small enough then the domain $D_k=D_k(\varepsilon,\theta,\delta)$ is invariant for $F$, for any $k\in\{0,\dots,d-1\}$, and that $F^j(x,y)\to 0$ for any $(x,y)\in D_k$, so by Proposition~\ref{parabola} the orbit of any point in $D_k(\varepsilon,\theta,\delta)$ eventually lies in $U_k(\varepsilon,\theta)$. The proof of the invariance of $D_k$ goes exactly as the the proof of the invariance of $U_k$ in Proposition~\ref{parabola}, taking into account that, since $\re a,\re b<0$, we can find a constant $\rho>0$ such that for $\varepsilon,\theta$ and $\delta$ small enough we have $|y_1|\le|y|\left(1-\rho|x^my^n|^d\right)\le|y|$ for all $(x,y)\in D_k$ and, if $n\ge1$, $|x_1|\le |x|\left(1-\rho|x^my^n|^d\right)\le|x|$. To show that $(x_j,y_j)\to0$ as $j\to\infty$ for all $(x_0,y_0)\in D_k$ we can also argue as in the proof of Proposition~\ref{parabola}: by the inequality above, we have that $|y_j|\le |y_0|\prod\limits_{l=0}^{j-1}\left(1-\rho\left|x_l^my_l^n\right|^d\right)$ and, if $n\ge1$, 
$|x_j|\le |x_0|\prod\limits_{l=0}^{j-1}\left(1-\rho\left|x_l^my_l^n\right|^d\right).$ Hence, since the product tends to 0 when $j\to\infty$, so do $x_j$ and $y_j$. These invariant domains $D_k$ are the ones found by Vivas in \cite[Lemma 1]{Viv}, and the reason to consider the smaller domains $U_k$ is that they are the ones where we will be able to construct an invariant function for $F$, as we will see next.
\end{remark}

In the rest of this section, we will show the existence of an invariant function for $F$ on each of the invariant domains $U_0,\dots, U_{d-1}$.
Since $F$ is close to the time-1 flow of the vector field 
$$x^My^N\Bigl( ax\parcial x+ by\parcial y\Bigr)$$
and the vector field $ax\partial_x+ by\partial_y$ has the Liouvillian first integrals $x^{-\eta b}y^{\eta a}$, 
$\eta\in\mathbb{C}^*$, our aim is to find an invariant function close to one of these first integrals,
for which we start defining a suitable branch $g(x,y)$ of $x^{d b}y^{-d a}$ on $U_k$, where $0\le k\le d-1$. From now on, if $z\in\mathbb{C}\backslash [-\infty,0]$ and $\lambda\in\mathbb{C}\backslash\mathbb{Z}$, we denote $z^{\lambda}=e^{\lambda\log z},$
where $\log$ is the main branch of the logarithm. Note that, since $m$ and $n$ are coprime if $n>0$,
 there exist $p,q\in\mathbb{N}$ such that $qm-pn=1$; in case $n=0$, we set $p=0$, $q=1$. Denote $\lambda=d(ap+bq)$ and define $g:U_k\to\C^*$ as
\begin{equation}\label{eq:defg}
g(x,y)=x^py^{q}(x^my^n)^{\lambda},
\end{equation}
which is well defined since $x^my^n$ belongs to $\mathbb{C}\backslash [-\infty,0]$ for all $(x,y)\in U_k$. Notice that in case $n=0$ we have that $g(x,y)=yx^{Mb}$. If $x$ and $y$ belong to a small sector bisected by $\R^+$ we have, 
in view of the identity $dma+dnb=-1$, that
$$g(x,y)=x^py^{q}(x^my^n)^{\lambda}= x^{p+m\lambda}y^{q+n\lambda}=x^{d b}y^{-d a},$$ 
so $g$ is a branch of $ x^{d b}y^{-d a}$ on $U_k$. 

\begin{proposition} \label{propsi} If $\varepsilon$ is small enough, then for each $k\in\{0,\dots, d-1\}$ there exists a function $\psi_k\in \mathcal{O}(U_k)$ which is invariant by $F$, i.e. $\psi_k\circ F=\psi_k$. Moreover, 
$\psi_k=ug$, where $g$ is the function defined above and $u\in \mathcal{O}(U_k)$ satisfies $ |u(x,y)-1|<1/2$ for all $(x,y)\in U_k$; in particular, $\psi_k(U_k)\subset\mathbb{C}^*$.
\end{proposition}
\proof We define $\psi_k$ as 
$$\psi_k(x,y)=\lim\limits_{j\to\infty}g(x_j,y_j), \ \ (x,y)\in U_k,$$
where $(x_j,y_j)=F^j(x,y)$. It is clear that this function, if well defined,  will  be invariant for $F$.
Let us show that it is well defined and holomorphic.
Using the expression of $F$ and equation~\eqref{eq:x^my^n}, we have that
$$\dfrac{x_1^py_1^q}{x^py^q}=1+(x^my^n)^d[ap+bq+O(x,y)]$$
and
$$\dfrac{(x_1^my_1^n)^\lambda}{(x^my^n)^\lambda}=1-(x^my^n)^d\left[\frac\lambda d+O(x,y)\right]$$
for all $(x,y)$, so 
\begin{equation}\label{propro}
\frac{g(x_1,y_1)}{g(x,y)}=1+\ell(x,y), \ \ \text{with } \ell(x,y)=(x^my^n)^dO(x,y).
\end{equation}
Since $(x_j,y_j)\in U_k$ for all $j$, we have that $|x_j|,|y_j| \le |x_j^my_j^n|^\gamma$ (in case $n=0$ the bound for $|x_j|$ is immediate, since $\gamma<1$), so
$$|\ell(x_j,y_j)|\le K|x_j^my_j^n|^{d+\gamma}$$
for some $K>0$ and therefore, by equation~\eqref{LFbounds0}, the product $\prod_{j\ge0}\frac{g(x_{j+1},y_{j+1})}{g(x_j,y_j)}$ converges uniformly
 for  $(x,y)\in {U_k}$ and defines a holomorphic function $u\in\mathcal{O}({U_k})$. Then $g(x_j,y_j)\to u(x,y)g(x,y)$ uniformly on $U_k$, so $\psi_k$ is well defined and holomorphic in ${U_k}$, and we have $\psi_k=ug$.
Note that the function $u$ is arbitrarily close to $1$ if we suppose $|x^my^n|$ 
to be small enough: if $n\in\mathbb{N}$ is  large enough the finite product 
$\prod_{0\le j\le n}\frac{g(x_{j+1},y_{j+1})}{g(x_j,y_j)}$ is arbitrarily uniformly close to $u$ and 
we have from \eqref{propro} that this finite product is arbitrarily uniformly close to 1 if 
$|x^my^n|$ is small enough. Therefore, we can assume that $\varepsilon$ is small enough so that 
$|u(x,y)-1|<1/2$ for all $(x,y)\in U_k$.
\qed

\section{Approximate Fatou coordinates}\label{approximate}

Fix $\varepsilon, \theta$ small enough so that Propositions~\ref{parabola} and \ref{propsi} hold. As a first approximation to Fatou coordinates on $U_k=U_k(\varepsilon, \theta)$ (i.e. conjugations with $(z,w)\mapsto(z+1,w)$) we consider the map 
$\phi_k\colon U_k\to \mathbb{C}\times \mathbb{C}^*$ defined as 
$$\phi_k(x,y)=\left(\frac{1}{(x^my^n)^d},\psi_k(x,y)\right),$$
where $\psi_k\in\mathcal{O}(U_k)$ is the invariant function given by Proposition~\ref{propsi}.

For $0<r<1$, consider the set $V=V(\varepsilon, \theta, r)$ defined by
$$V=\left\{(z,w)\in\C^2:|z|>\varepsilon^{-1}, \ |\arg z|<\theta, \ |w|<r|z|^{\frac{-\re b}{m}-\frac{\gamma}{dm}}\right\}$$
in case $n=0$ and
$$V=\left\{(z,w)\in\C^2:|z|>\varepsilon^{-1}, \ |\arg z|<\theta, \ r^{-1}|z|^{\frac{\re a}{n}+\frac\gamma{dn}}<|w|<r|z|^{\frac{-\re b}{m}-\frac{\gamma}{dm}}\right\}$$
in case $n\ge1$. Notice that $V$ is homeomorphic to $\C^2$ if $n=0$ and to $\C\times\C^*$ if $n\ge1$.

\begin{remark}\label{rk:Vfundamentaldomain}
If an orbit $F^j(x,y)=(x_j,y_j)$ converges to 0, then we have by Proposition~\ref{parabola} that $(x_j,y_j)\in U_k$ for some $k$ and for all $j\ge j_0$, and it also holds that $\phi_k(x_j,y_j)\in V$  if $j$ is large enough: if we set $(z_j,w_j)=\phi_k(x_j,y_j)$ for all $j\ge j_0$, then clearly $|z_j|>\varepsilon$ and $|\arg z_j|<\theta$; moreover, since $\psi_k$ is invariant for $F$ we have that $w_j$ is a nonzero constant for all $j$ while $z_j\to +\infty$, so for $j$ large enough we get $|w_j|<r|z_j|^{\frac{-\re b}{m}-\frac{\gamma}{dm}}$ and, if $n\ge1$, $|w_j|>r^{-1}|z_j|^{\frac{\re a}{n}+\frac\gamma{dn}}$, so $(z_j,w_j)\in V$. 
\end{remark}

\begin{lemma}\label{lem:phibiholom}
If $r$ is sufficiently small, then $V\subset \phi_k(U_k)$ for every $k$ and $\phi_k:\phi_k^{-1}(V)\to V$ is a biholomorphism.
\end{lemma}
\begin{proof}
Without loss of generality, we assume $k=0$. Let $g$ be the function defined by~\eqref{eq:defg}. Using the fact that $qm-pn=1$ and $adm+bdn=-1$, a straightforward computation shows that the map $\varphi:U_0\to\C^2$ given by 
$$\varphi(x,y)=\left(\frac{1}{(x^my^n)^d},g(x,y)\right)$$
is injective and its inverse is given by 
$$\varphi^{-1}(z,w)=\left(z^aw^{-n},z^bw^{m}\right).$$
Consider a point $(z_0,w_0)\in V$, and set
$$A_{z_0}=\left\{(x,y)\in\C^2: x^my^n=z_0^{-1/d}, \ |\epsilon x|<|z_0|^{-\gamma/d}, \ |y|<|z_0|^{-\gamma/d}\right\}.$$
Notice that $A_{z_0}\subset U_0$ and $\overline{A_{z_0}}$ is a Riemann surface  with boundary. If we set $x=z_0^aw_0^{-n}$ and $y=z_0^bw_0^{m}$, then $x^my^n=z_0^{-1/d}$ and $g(x,y)=w_0$. Moreover, $|y|=|z_0^bw_0^m|<r^m|z_0^b||z_0|^{-\re b-\gamma/d}\le r^me^{|b|\theta}|z_0|^{-\gamma/d}$ and, if $n\ge1$, analogously $|x|=|z_0^aw_0^{-n}|<r^ne^{|a|\theta}|z_0|^{-\gamma/d}$, so $(x,y)\in A_{z_0}$ if $r$ is small enough. Therefore, since $\varphi$ is injective, $g|_{A_{z_0}}$ assumes the value $w_0$ once.
If we show that 
\begin{equation}\label{eq:rouche}
|\psi_0(x,y)-g(x,y)|<|g(x,y)-w_0|\quad \textrm{whenever} \ (x,y)\in\partial A_{z_0}
\end{equation}
then, by Rouché's theorem, the function $\psi_0|_{A_{z_0}}$ will also assume the value $w_0$ exactly once, showing that $V\subset\phi_0(U_0)$ and that $\phi_0$ is injective in $\phi_0^{-1}(V)$. Let us prove that the inequality \eqref{eq:rouche} holds. 
In case $n\ge1$, the boundary $\partial A_{z_0}$ of $A_{z_0}$ is composed by two connected components, 
\begin{align*}
\partial_1 A_{z_0}&=\left\{(x,y)\in\C^2:x^my^n=z_0^{-1/d}, \ |x|=|z_0|^{-\gamma/d} \right\}\\
\partial_2 A_{z_0}&=\left\{(x,y)\in\C^2:x^my^n=z_0^{-1/d}, \ |y|=|z_0|^{-\gamma/d}\right\},
\end{align*}
whereas in case $n=0$ we have $\partial A_{z_0}=\partial_2 A_{z_0}$.
Consider a point $(x,y)\in \partial A_{z_0}\subset U_0$. Since $x^my^n=z_0^{-1/d}$,
it follows from the computation of $\varphi^{-1}$ that 
$$g(x,y)^n= 
x^{-1}z_0^a\quad\text{ and }\quad g(x,y)^m=yz_0^{-b}.$$ 
We suppose first that  $(x,y)\in \partial_2 A_{z_0}$. Then
$$|g(x,y)|= |y|^{\frac{1}{m}}|z_0^{-\frac bm}| 
= |z_0|^{-\frac{\gamma}{dm}}|z_0^{-\frac bm}|
\ge e^{-\frac{|b|\theta}m}|z_0|^{-\frac{\re b}m-\frac{\gamma}{dm}}$$
so $|g(x,y)|>e^{-\frac{|b|\theta}m}r^{-1} |w_0|$.
Then 
$$|g(x,y)-w_0|\ge |g(x,y)|-|w_0|>\Bigl(1-e^{\frac{|b|\theta}m}r\Bigr)|g(x,y)|\ge \frac{1}{2}|g(x,y)|$$
if $r$ is small enough. 
This relation, together with the fact that 
$$|\psi_0(x,y)-g(x,y)|<\frac12|g(x,y)|$$
for all $(x,y)\in U_0$, as was shown in Proposition~\ref{propsi}, implies \eqref{eq:rouche}. Analogously, if 
$(x,y)\in \partial_1 A_{z_0}$ (so $n\ge1$) then
$$|g(x,y)|=|x|^{-\frac{1}{n}}|z_0^{\frac an}|= |z_0|^{\frac{\gamma}{dn}}|z_0^{\frac{a}{n}}|
\le e^{\frac{|a|\theta}{n}}|z_0|^{\frac{\re a}{n}+\frac{\gamma}{dn}}$$
so $|g(x,y)|< e^{\frac{|a|\theta}{n}}r|w_0|$ and hence
$$|g(x,y)-w_0|\ge|w_0|- |g(x,y)|>\Bigl(e^{-\frac{|a|\theta}{n}}r^{-1}-1\Bigr)|g(x,y)|
\ge \frac{1}{2}|g(x,y)|$$
if $r$ is sufficiently small, which again implies~\eqref{eq:rouche}.
\end{proof}

By Lemma~\ref{lem:phibiholom}, if $r$ is small enough then $\phi_k^{-1}$ is well defined on $V=V(\varepsilon,\theta, r)$ and takes values in $U_k$, so
$$\tilde{F}= \phi_k\circ F\circ \phi_k^{-1}$$ 
is  well defined on $V$.  Since $\psi_k$ is invariant by $F$, we can express $\tilde{F}(z,w)=(f(z,w),w)$. Then, if we write
$\phi_k^{-1}(z,w)=(x,y)\in U_k$ and $F(x,y)=(x_1,y_1)$, we have from equation~\eqref{eq:x^my^n} that
$$f(z,w)=\frac{1}{(x_1^my_1^n)^d}=\frac{1}{(x^my^n)^d}+1+\tilde{h}(x,y),$$
where $\tilde{h}(x,y)=O(x,y)$. That is, 
$$f(z,w)=z+1+{h}(z,w),\quad \text{ where } \ {h}(z,w)=\tilde{h}(x,y).$$
Since $(x,y)\in U_k$, we have that $|x|,|y|\le|z|^{-\gamma/d}$, so there is a constant $K>0$ such that 
$$|h(z,w)|<K|z|^{-\gamma/d}$$
for all $(z,w)\in V$. In particular, if $\varepsilon$ is small enough we have that $|f(z,w)|\ge |z|+\frac 12$. Moreover, we have the following

\begin{lemma}\label{disco} 
If $\varepsilon$, $\theta$ and $r$ are sufficiently small then $\tilde F(V)\subset V$ and there exists a constant $K'>0$ such that
$$\left|\frac{\partial h}{\partial z} (z,w)\right|<K'|z|^{-1-\gamma/d}$$
for every $(z,w)\in V.$ 
\end{lemma}
\begin{proof}
Consider a point $(z,w)\in V$. Since $(f(z,w),w)\in \phi_k(U_k)$, it is clear that $|f(z,w)|>\varepsilon^{-1}$ and $|\arg (f(z,w))|<\theta$ and, since 
$|w|<r|z|^{\frac{-\re b}{m}-\frac{\gamma}{dm}}$ and $|f(z,w)|>|z|$, we also have that 
$$|w|<r|f(z,w)|^{\frac{-\re b}{m}-\frac{\gamma}{dm}}.$$
Analogously, if $n\ge1$, since $|w|>r^{-1}|z|^{\frac{\re a}{n}+\frac\gamma{dn}}$ and $|f(z,w)|>|z|$, we have $|w|>r^{-1}|f(z,w)|^{\frac{\re a}{n}+\frac\gamma{dn}}$, so $\tilde F(z,w)\in V$. To prove the bound for $\frac{\partial h}{\partial z}$ let us first show that there exists $\rho>0$ such that if $(z_0,w_0)\in V\left(\varepsilon/2,\theta/2,r/2\right)$ and $|z-z_0|<\rho|z_0|$ then $(z,w_0)\in V(\varepsilon, \theta, r)$.
Consider $(z_0,w_0)\in V\left(\varepsilon/2,\theta/2,r/2\right)$ and assume that $|z-z_0|<\rho|z_0|$.  Then 
$$|z|>(1-\rho )|z_0|>(1-\rho)2\varepsilon^{-1},$$
so $|z|>\varepsilon^{-1}$ for $\rho$ sufficiently small.
Since $|z/z_0-1|<\rho$, we have $|\arg(z/z_0)|<\arcsin\rho$, so
$$|\arg z|\le |\arg z_0 |+\arcsin\rho<\frac\theta 2+\arcsin\rho,$$
hence $|\arg z|<\theta$ if $\rho$ is small enough.
Since $|z_0|<(1-\rho)^{-1}|z|$, it follows that
$$|w_0|<\frac r2|z_0|^{\frac{-\re b}{m}-\frac{\gamma}{dm}}<\frac r2\left[(1-\rho)^{-1}|z|\right]^{\frac{-\re b}{m}-\frac{\gamma}{dm}},$$
so $|w_0|<r|z|^{\frac{-\re b}{m}-\frac{\gamma}{dm}}$ if $\rho$ is small enough
and, if $n\ge1$, 
$$|w_0|>2r^{-1}|z_0|^{\frac{\re a}{n}+\frac\gamma{dn}}>2r^{-1}\left[(1-\rho)^{-1}|z|\right]^{\frac{\re a}{n}+\frac\gamma{dn}},$$
so $|w_0|>r^{-1}|z|^{\frac{\re a}{n}+\frac\gamma{dn}}$ if $\rho$ is small enough. Hence, $(z,w_0)\in V(\varepsilon, \theta, r)$ if $\rho$ is small enough. 
Take a point $(z_0,w_0)\in V\left(\varepsilon/2,\theta/2,r/2\right)$.  If $D\subset \mathbb{C}$ is
the disc of radius $\rho|z_0|$ centered at $z_0$, then $D\times\{w_0\}$ is contained in $V(\varepsilon, \theta, r)$, so the function
$$h_{w_0}\colon z\in D \mapsto h(z,w_0)$$ is well defined and 
$$|h_{w_0}(z)|<K|z|^{-\gamma/d}<K(1-\rho)^{-\gamma/d}|z_0|^{-\gamma/d}.$$ 
Thus, it follows from Cauchy's inequality that
\begin{equation*}
\left|\frac{\partial h}{\partial z}(z_0,w_0)\right|=\left|(h_{w_0})'(z_0)\right|\le K(1-\rho)^{-\gamma/d}|z_0|^{-\gamma/d}(\rho|z_0|)^{-1}=K'|z_0|^{-1-\gamma/d}.\qedhere
\end{equation*}
\end{proof}

\section{Existence of Fatou coordinates}
Consider $\varepsilon$, $\theta$ and $r$ small enough so that Lemma~\ref{disco} holds in $V=V(\varepsilon,\theta,r)$. In this section we construct a biholomorphism
 $\Phi\colon V\to W\subset\mathbb C^2$, with $W\subset\C\times\C^*$ if $n\ge1$, conjugating $\tilde{F}$ with 
the map $(z,w)\mapsto (z+1,w)$.  
Since $\tilde{F}\colon V\to V$ is written $\tilde{F}(z,w)=(f(z,w),w)$,  each function $f_{w}\colon z\mapsto f(z,w)$ maps the domain $V_w=\{z\in\C:(z,w)\in V\}$ into itself. Thus, we start considering $w\in\C$ fixed ($w\in\mathbb{C}^*$ if $n\ge1$) and we 
find a conjugation of $f_w$ with the map $z\mapsto z+1$ following the ideas in \cite[Lemma 10.10]{Mil}. From the definition of $V$
we have
$$V_w=\left\{z\in\mathbb{C}\colon |z|>R_w,\; |\arg(z)|<\theta \right\},$$
where 
$$R_w=\max \left\{\varepsilon^{-1},\; (r^{-1}|w|)^{-dm/(d\re b+\gamma)},
\; \epsilon(r|w|)^{dn/(d\re a+\gamma)}\right\}.$$

Take a base point $p\in V_w$. The conjugation of $f_w$ with $z\mapsto z+1$  will be constructed as the limit of the functions
$$\beta_j(z)= f^j_w(z)-f^{j}_w(p),\quad j\in\mathbb{N}.$$
In order to simplify the proof of the convergence of these functions we assume $p$ to be large enough so that
for all $z\in V_w$  the euclidean segment $[z,p]$ is contained in $V_w$, which is possible because $\theta<\pi/2$. Since $|f_w(p)|\ge |p|+1/2$, the sequence
$|f_w^j(p)|$ is increasing, hence the property above also holds for $f_w^j(p)$. In  particular, we 
have, for all $z\in V_w$ and all $j\in\mathbb{N}$,
$$[f_w^j(z),f_w^j(p)]\subset V_w.$$
Since $f_w(z)=z+1+h(z,w)$ and $\left|\frac{\partial h}{\partial z}(z,w)\right|<K'|z|^{-1-\gamma/d}$, it follows from the mean value inequality that, if $[z_1,z_2]\subset  V_w$, then
\begin{align*}
\left|\frac{f_w(z_1)-f_w(z_2)}{z_1-z_2}-1\right|=
\left|\frac{h(z_1,w)-h(z_2,w)}{z_1-z_2}\right|\le \max\limits_{z\in[z_1,z_2]}\frac{K'}{|z|^{1+\gamma/d}}.
\end{align*}
Since the angle between $z_1\mathbb{R}^+$ and  $z_2\mathbb{R}^+$ is bounded by $2\theta<\pi$, there is a constant $\tau>0$ depending only on $\theta$ such that $\min\{|z|:z\in[z_1,z_2]\}\ge \tau\min\{|z_1|,|z_2|\}$, so
$$\left|\frac{f_w(z_1)-f_w(z_2)}{z_1-z_2}-1\right|\le \frac{K'}{\left(\tau\min\{|z_1|,|z_2|\}   \right)^{1+\gamma/d}}.$$
In particular, setting $z_1=f_w^j(z)$ and $z_2=f_w^j(p)$ we obtain 
$$\left|\frac{\beta_{j+1}(z)}{\beta_j(z)}-1\right|=\left|\frac{f_w^{j+1}(z)-f_w^{j+1}(p)}{f_w^{j}(z)-f_w^{j}(p)}-1\right|
\le \frac{K'}{\left(\tau\min\{|f_w^{j}(z)|,|f_w^{j}(p)|\}\right)^{1+\gamma/d}}$$
for all $z\in V_w$ and $j\in\mathbb{N}$.
Therefore, since $|f_w^{j}(z)|\ge j/2$ and $|f_w^{j}(p)|\ge j/2$, we
obtain
$$\left|\frac{\beta_{j+1}(z)}{\beta_j(z)}-1\right|\le 
K'(2\tau^{-1})^{1+\gamma/d}j^{-1-\gamma/d}$$
for all $z\in V_w$ and $j\in\mathbb{N}$. 
This shows that the product $\prod \frac{\beta_{j+1}(z)}{\beta_j(z)}$ is uniformly convergent in $V_w$ and therefore $\beta_j$ converges uniformly to a function $\beta_w\in \mathcal{O}(V_w)$.
Let us  show that $\beta_w$ is one to one and  conjugates $f_w$ with  the map $z\mapsto z+1$. 
Since $f_w(z)=z+1+h(z,w)$ and $|h(z,w)|<K|z|^{-\gamma/d}$ we have 
$$\left| f_w^{j+1}(p)-f_w^j(p) -1\right|= \left| h(f_w^j(p),w)\right|\le \frac{K}{|f_w^j(p)|^{\gamma/d}}\to 0
\quad \textrm{as } j\to\infty.$$
Thus, since $\beta_j(f_w(z))=\beta_{j+1}(z)+f_w^{j+1}(p)-f_w^j(p)$, taking $j\to\infty$ we obtain
$$\beta_w(f_w(z))=\beta_w(z)+1,\quad z\in V_w.$$
Finally, since $\beta_j$ is injective for all $j$ and $\beta_w$ is not constant, we conclude that $\beta_w$ is  injective. 
 
Now, as in \cite[Lemma 10.11]{Mil}, we prove the following additional property of $\beta_w$:
\begin{align}\label{cuber}\bigcup_{j\in\mathbb{N}} \left(\beta_w(V_w)-j\right)=\mathbb{C}.
\end{align}
We show first that $\lim\limits_{z\to \infty}\frac{\beta_w(z)}{z}=1$. Since $\beta_j$ tends uniformly to $\beta_w$, for some $l\in\mathbb{N}$ we have that  $|\beta_w-\beta_l|$ is bounded, whence
$$|\beta_w-f_w^l|\le |\beta_w-\beta_l|+|\beta_l-f_w^l|$$ is bounded. Then, since $f_w(z)=z+1+h(z,w)$ and $|h(z,w)|<K|z|^{-\gamma/d}$, 
$$\lim\limits_{z\to \infty}\frac{\beta_w(z)}{z}=\lim\limits_{z\to \infty}\frac{f_w^l(z)}{z} =1.$$
Consider $\zeta\in\mathbb{C}$. In order to prove \eqref{cuber} we will show that for $j\in\mathbb{N}$ large enough the point $\zeta_j=\zeta+j$ belongs to $\beta_w(V_w)$. Since $V_w$ is essentially a sector
of opening $2\theta$, if we take a positive number $\rho<\sin \theta$, it is not difficult to see that,
for $j$ large enough,
the closed disc $D_j$ of radius $r_j = \rho | \zeta_j|$ centered at $\zeta_j$ is contained in $V_w$. By Rouché's theorem, if
$$|\beta_w(z)-z|<r_j$$
for all $z\in \partial D_j$, then $\zeta_j\in\beta_w(D_j)\subset\beta_w(V_w)$. Since $\beta_w(z)/z\to 1$ when $z\to\infty$, for $z$ large enough we have $|\beta_w(z)-z|<\tau|z|$, where $\tau>0$ is taken such that ${\tau}(1+\rho)<\rho$. Then, if $z\in \partial D_j$ and $j$ is large enough,
$$|\beta_w(z)-z|<\tau|z|\le\tau(|\zeta_j|+r_j)=\tau(1+\rho)|\zeta_j|<\rho|\zeta_j|=r_j$$
and \eqref{cuber} follows.

\strut
 
The function $\beta_w$ that we have constructed depends on the choice of the base point $p\in V_w$, but its derivative does not, as we can check from the definition of $\beta_j$:
$$(\beta_w)'(z)=\lim\limits_{j\to\infty}(f_w^j)'(z).$$
It is easy to see that the same choice of $p$ also works for any  $w'$ in a neighborhood of $w$ and the function $\beta_{w'}$ will depend holomorphically on $w'$. That is, $\beta_{w'}(z)$ is a holomorphic function of $(z,w')$.  Thus, we can find an open covering
$\mathbb{C}= \bigcup_{i\in I} W_i$ if $n=0$ or $\mathbb{C}^*= \bigcup_{i\in I} W_i$ if $n\ge1$ and, for each $i\in I$, a holomorphic function 
$$\beta_i(z,w),\quad \text{ for } z\in V_w,\; w\in W_i$$ 
such that, for each $w\in W_i $, the map $z\in V_w\mapsto \beta_i(z,w)\in\mathbb{C}$ is univalent and satisfies
$\beta_i(f(z,w),w)=\beta_i(z,w)+1.$ Moreover, from the observation above the partial derivative 
$\frac{\partial \beta_i}{\partial z}$ does not depend on $i\in I$, that is,
$$\frac{\partial \beta_i}{\partial z}(z,w)=\frac{\partial \beta_j}{\partial z}(z,w),\quad \text{ for } z\in V_w,\; w\in W_i\cap W_j,\;i,j\in I.$$
Therefore, if $W_i\cap W_j\neq \emptyset$, there is a function $g_{ij}\in\mathcal{O}(W_i\cap W_j)$
such that 
\begin{align}\label{torta}
\beta_j(z,w)-\beta_i(z,w)= g_{ij}(w), \quad \text{ for } z\in V_w,\; w\in W_i\cap W_j,
\end{align}
hence $g_{ij}+g_{jk}+g_{ki}=0$ on $W_i\cap W_j\cap W_k$, for $i,j,k\in I$.
Then, since the first Cousin problem can be solved in $\C$ and $\mathbb{C}^*$,  there exist  functions 
$g_i\in\mathcal{O}(W_i)$, $i\in I$, such that $g_{ij}=g_i-g_j$ on $W_i\cap W_j$ and it follows from
\eqref{torta} that 
$$\beta_j(z,w)+g_j(w)=\beta_i(z,w) +g_{i}(w), \quad\text{ for } z\in V_w,\; w\in W_i\cap W_j.$$
Therefore we can define a global function $\beta\in\mathcal{O}(V)$ by 
$$\beta(z,w)=\beta_i(z,w)+g_i(w),\quad\text{ for } z\in V_w,\;w\in W_i$$ and we can see that, for each $w\in \mathbb{C}^*$, the map $$z\in V_w\mapsto \beta(z,w)\in\mathbb{C}$$ is univalent and $\beta(f(z,w),w)=\beta(z,w)+1$ for every $(z,w)\in V$. Now it is easy to check that
the holomorphic function $$\Phi(z,w)= (\beta(z,w),w), \quad (z,w)\in V$$
is univalent and satisfies $\Phi\circ \tilde{F}(z,w)=\Phi(z,w)+(1,0)$ for every $(z,w)\in V$. Moreover, $W=\Phi(V)$ satisfies 
\begin{align}
\label{stabilo}\bigcup_{j\in\N} \left[W-(j,0)\right]=\C^2 \ \text { if } n=0;\quad  \bigcup_{j\in\N} \left[W-(j,0)\right]=\C\times\C^* \ \text { if } n\ge1.
\end{align}
To show this property, consider a point $(z_0,w_0)\in \C^2$, with $(z_0,w_0)\in\mathbb{C}\times \mathbb{C}^*$ if $n\ge1$. If $w_0\in W_i$, we have
$$\beta(z,w_0)=\beta_i(z,w_0)+g_i(w_0)=\beta_{w_0}(z)+g_i(w_0)$$
for all $z\in V_{w_0}.$ By \eqref{cuber} there exist $z\in V_{w_0}$ and $j\in\mathbb{N}$ such that $[z_0-g_i(w_0)]+j=\beta_{w_0}(z)$, thus
$$\Phi(z,w_0)=(\beta(z,w_0),w_0)=(z_0+j,w_0)$$ and therefore
$(z_0,w_0)\in W-(j,0)$. 

\section{The flower theorem}\label{secfatou}
In this section we complete the proof of Theorem~\ref{FatouFlower}. As we mentioned in Section~\ref{constructing}, we will do it first in the case where 
$F$ has the form \eqref{eq:cornerdiffeo} with $N\ge0$ and satisfying condition \eqref{eq:hipononcorner}, and at the end of the section we will deal with the case where $F$ is of the form \eqref{eq:noncornerdiffeo}.

Consider $\varepsilon, \delta, r>0$ and $\theta\in\left(0,\frac{\pi}{2}\right)$ small enough so that Proposition~\ref{parabola}, Remark~\ref{rk:largerdomains} and Lemma~\ref{disco} hold. Consider the set
$$\widetilde S(\varepsilon,\theta)\subset\C$$ 
defined in Section~\ref{constructing}, and let
$\widetilde S_k$ be, for $k\in\{0,\dots,d-1\}$, the connected component of $\widetilde S(\varepsilon,\theta)$ bisected by $e^{2\pi i k/d}\R^+$, which is a sectorial domain of opening $(\pi+2\theta)/d$. For $0<\delta'\le\delta$ and for each $k\in\{0,\dots,d-1\}$, let $\widetilde D_k=\widetilde D_k(\delta')$ be the set defined by 
$$\widetilde D_k=\left\{(x,y)\in\C^2:x^my^n\in \widetilde S_k, |\epsilon x|<\delta',|y|<\delta'\right\}.$$
Since $\gcd(m,n)=1$, it is easy to see that the sets $\widetilde D_k$ are connected.
\begin{lemma}\label{tati} If $\delta'$ is small enough, then for any point $p\in \widetilde D_k(\delta')$ there exists $j\ge 0$
such that $F^{j}(p)\in \phi_k^{-1}(V)$. 
\end{lemma}
\begin{proof}
Without loss of generality, we assume $k=0$. By Remarks~\ref{rk:largerdomains} and \ref{rk:Vfundamentaldomain}, it is enough to prove that if $\delta'>0$ is small enough then for any $(x,y)\in\widetilde D_0(\delta')$  there exists $j\in\N$ such that $(x_j,y_j)=F^j(x,y)\in D_0(\varepsilon,\theta,\delta)$. From equation~\eqref{eq:x^my^n},
arguing as in Theorem~\ref{the:flowerdim1} we find constants
$c>1$ and $C>0$, depending only on $\varepsilon$ and $\theta$, such that  if $(x,y)$ satisfies that $x^my^n\in \widetilde S_0$ and $|\epsilon x_j|,|y_j|<\delta$ for all $j$  then  $x_j^my_j^n\in\widetilde S_0$ and $|x_j^my_j^n|^d\le c|x^my^n|^d$ for all 
 $j$ and $|\arg(x_j^my_j^n)|<\theta/d$ for every $j\ge C/|x^my^n|^{d}$. Notice that we can find a constant $\rho>0$ such that if $x,y\in\C^*$ are small enough then
$$|y_1|=|y|\left|1+b(x^my^n)^d+(x^my^n)^dO(x,y)\right|\le |y|\left(1+\rho|x^my^n|^d\right).$$
Set $K=\sup_{t\in(0,\varepsilon)}(1+\rho ct^d)^{C/t^d+1}$ and consider $\delta'\le \delta/K$. Take $(x,y)\in \widetilde D_0(\delta')$ and set $j_0=\lceil C/|x^my^n|^{d}\rceil$.
We have that, for all $j\le j_0$,
$$|y_j|\le |y|\prod_{l=0}^{j_0-1}\left(1+\rho|x_l^my_l^n|^d\right)\le |y|\left(1+\rho c|x^my^n|^d\right)^{j_0},$$
so $|y_j|<\delta' K\le\delta$ for all $j\le j_0$. If $n\ge1$, analogously $|x_j|\le\delta$ for all $j\le j_0$. Since we also have $|\arg(x_j^my_j^n)|<\theta/d$, it follows that $(x_{j_0},y_{j_0})\in D_0(\varepsilon,\theta,\delta)$.
\end{proof}
We can now complete the proof of Theorem~\ref{FatouFlower}. Consider $\delta'>0$ such that Lemma~\ref{tati} holds. For each $k\in\{0,\dots, d-1\}$, we define
$$\Omega^+_k=\bigcup_{j\ge0}F^j\left(\widetilde D_k(\delta')\right).$$ We see that $\Omega^+_k$ is
connected, invariant by $F$ and attracted to $U_k$, by Lemma~\ref{tati}.
Clearly $\widetilde D_k(\delta')\subset \Omega^+_k$, and from the proof of Lemma~\ref{tati} we also have that
$$\Omega^+_k\subset\widetilde D_k(K\delta').$$
It is easy to see that the diffeomorphism $F^{-1}$ is also of the form
\eqref{eq:cornerdiffeo}, with the same 
pair $(M,N)$ and $(-a,-b)$ instead of $(a,b)$. Thus, if we work with  $F^{-1}$
instead of $F$, the previous construction allows us to find connected open sets $\Omega^-_0, \dots ,\Omega^-_{d-1}$ where each $\Omega^-_k$ is defined by 
$$\Omega^-_k=\bigcup_{j\ge0}F^{-j}\left(\widetilde D^-_k(\delta')\right),$$
with
$$\widetilde D^-_k(\delta')=\left\{(x,y)\in\C^2:x^my^n\in \widetilde S^-_k, |\epsilon x|<\delta',|y|<\delta'\right\},$$
being $\widetilde S^-_k$, $0\le k\le d-1$, the connected components of $\widetilde S^-(\varepsilon,\theta)$ (see Remark~\ref{rk:repelling}). In each $\Omega^-_k$, we have $F^{-j}\to0$ and $D_k^-(\delta')\subset\Omega^-_k\subset D^-_k(K\delta')$.

Since the opening of the components of $\widetilde S(\varepsilon,\theta)$ and $\widetilde S^-(\varepsilon,\theta)$ is greater than $\pi/d$, it is clear that  the domains $\Omega^+_0,\dots,\Omega^+_{d-1},\Omega^-_0,\dots,\Omega^-_{d-1}$, together with the fixed point set  $\{xy^\epsilon=0\}$, cover the open set
$$\left\{(x,y)\in\C^2:|x^my^n|<\varepsilon^{1/d}, |\epsilon x|<\delta', |y|<\delta'\right\},$$ 
so assertions \ref{item:parabolicdomains} and \ref{item:coverneigh} of Theorem~\ref{FatouFlower} are proved. 

For each $k$, set $\varphi^+_k=\Phi\circ\phi_k:\phi_k^{-1}(V)\to \mathbb C^2$, which is a univalent holomorphic map conjugating $F$ with the map $(z,w)\mapsto (z+1,w)$. It is straightforward to extend $\varphi^+_k$ as a biholomorphism 
$$\varphi^+_k\colon \Omega^+_k\to W_k^+\subset\C^2,$$
with $ W_k^+\subset\C\times\C^*$ if $n\ge1$, defining, for each $p\in\Omega_k^+$,  
$$\varphi^+_k(p)= \varphi^+_k(F^j(p))-(j,0)$$
for any $j\ge 0$  such that $F^j(p)\in \phi_k^{-1}(V)$. This shows assertion~\ref{tercero1+} of Theorem~\ref{FatouFlower} for the domains $\Omega^+_0,\dots,\Omega^+_{d-1}$; property~\ref{tercero1.5+} follows from \eqref{stabilo}. We can proceed analogously with the sets $\Omega^-_0,\dots,\Omega^-_{d-1}$ and this finishes the proof of Theorem~\ref{FatouFlower} for $F$ of the form \eqref{eq:cornerdiffeo} with $N\ge 0$.

Suppose now that $F$ is of the form \eqref{eq:noncornerdiffeo} satisfying \eqref{eq:hipononcorner}. As we saw in Section~\ref{constructing}, after the sectorial change of coordinates~\eqref{sectorialchange} in $\widetilde S_k(\varepsilon,\theta)\times \C$ we can write
$$F(x,y)=\left(x+x^{M+1}\left[a+O_1(x,y)\right], y+x^My\left[b+O_1(x,y)\right]\right).$$
The key point to note is that all the constructions we made in the previous sections to obtain the invariant sets $\Omega^+_k$ for a map $F$ of the form
\eqref{eq:cornerdiffeo} with $N=0$ were performed in $\widetilde S_k(\varepsilon,\theta)\times \C$, and all the calculations involved work analogously if we have $O_1(x,y)$ instead of $O(x,y)$. So the domains $\Omega^+_k$ can be defined in the same way, but in sectorial coordinates depending on $k$, and the same holds for $\Omega_k^-$. Assertions \ref{item:parabolicdomains} and \ref{tercero+} of Theorem~\ref{FatouFlower}, since referred to a fixed $k\in\{0,\dots, M-1\}$, follow exactly as above if we work in the corresponding sectorial coordinates.  For the proof of assertion \ref{item:coverneigh}, it is enough to show that each $\Omega^+_k$ contains a set of the form $\{(x,y)\in\C^2:x\in\widetilde S_k,|y|<\delta''\}$ in the original coordinates, for some $\delta''>0$, and the same for each $\Omega_k^-$.  In the original coordinates, $\Omega^+_k$ is actually given by a set of the form $\sigma_k(\Omega^+_k)$,
where $\sigma_k(x,y)=(x,y+u_k(x))$ is the inverse of the change of coordinates \eqref{sectorialchange}. Since $\Omega^+_k$ contains $\{(x,y)\in\C^2:x\in\widetilde S_k,|y|<\delta'\}$ and $|u_k(x)|\le K|x\log x|$ for some $K>0$, up to reducing $\varepsilon$ if necessary it is easy to see that $\sigma_k(\Omega^+_k)$ contains the set $\{(x,y)\in\C^2: x\in\widetilde S_k,|y|<\delta''\}$ for some $\delta''>0$.  Clearly the analogous property holds for each $\Omega_k^-$, so Theorem~\ref{FatouFlower} is proved.

\section{Proof of Theorem~\ref{saddlebehavior}}
Suppose first that $F$ is of the form \eqref{eq:cornerdiffeo}. As in Section~\ref{constructing}, up to a linear change of coordinates we can write $F$ as
$$F(x,y)=\left(x+x^{M+1}y^N\left[a+O(x,y)\right], y+x^My^{N+1}\left[b+O(x,y)\right]\right),$$
with $M\ge1$, $N\ge 1$ and $aM+bN=-1$, so the hypothesis of Theorem~\ref{saddlebehavior} means that either $\re a>0$ or $\re b>0$. We assume without loss of generality that $\re b>0$. As in the previous sections, put $d=\gcd(M,N)$ and set $m=M/d$, $n=N/d$. As in Leau-Fatou flower theorem, from equation~\eqref{eq:x^my^n} we have that there exists $\delta>0$ such that for $\varepsilon$ small enough we have that the orbit $(x_j,y_j)$ of any point
$$(x,y)\in \mathcal{U}=\left\{(x,y)\in\C^2: |x^my^n|<\varepsilon, |x|<\delta, |y|<\delta\right\}$$ 
either leaves $\mathcal U$ or satisfies $(x_j^my_j^n)^d\to 0$ along $\R^+$ and $\lim_{j\to\infty}j(x_j^my_j^n)^d=1$.
Consider a point $(x,y)\in\mathcal{U}$ outside the fixed set $\{xy=0\}$ and let us show that there exists $j\in\N$ such that $(x_j,y_j)\not\in\mathcal{U}$. Assume by contradiction that $(x_j,y_j)\in\mathcal{U}$ for all $j$, so $(x_j^my_j^n)^d\to 0$
along $\R^+$ and $\lim_{j\to\infty}j(x_j^my_j^n)^d=1$. 
Thus, since $\re b>0$, for some $j_0\ge0$ and some $\nu>0$ we have that 
$|y_{j+1}|\ge |y_j|\left(1+\nu |x_j^my_j^n|^d\right)$
for all $j\ge j_0$ and therefore
$$|y_j|\ge |y_{j_0}|\prod_{l=j_0}^{j-1}\left(1+\nu |x_l^my_l^n|^d\right)$$
for all $j\ge j_0+1$, whence, since
$\lim_{j\to\infty}j(x_j^my_j^n)^d=1$, we conclude that $|y_j|\to\infty$, which is a contradiction. In the same way, up to reducing $\delta$ and $\varepsilon$,  we can prove that the negative orbit of 
any $(x,y)\in\mathcal{U}\backslash \{xy=0\}$ leaves $\mathcal{U}$.

Suppose now that $F$ is of the form \eqref{eq:noncornerdiffeo}.  Again as in Section~\ref{constructing}, up to a linear change of coordinates we have 
\begin{align*}
F(x,y)=\left(x+x^{M+1}\left[-1/M+O(x,y)\right], y+x^M\left[by+O(x,y^2)\right]\right),
\end{align*} where $\re b>0$. As above, from the equation 
$x_1=x+x^{M+1}\left[-1/M+O(x,y)\right]$ we conclude that there exists 
$\delta>0$ such that for $\varepsilon$ small enough the orbit $(x_j,y_j)$ of any point
$$(x,y)\in \mathcal{U}=\left\{(x,y)\in\C^2: |x|<\varepsilon, |y|<\delta\right\}$$ 
either leaves $\mathcal U$ or satisfies $x_j^M\to 0$ along $\R^+$ and $\lim_{j\to\infty}jx_j^M=1$.
Consider a point $(x,y)\in\mathcal{U}$ outside the fixed set $\{x=0\}$ and suppose
that $(x_j,y_j)\in\mathcal{U}$ for all $j\in\mathbb N$, so $x_j^M\to 0$
along $\R^+$ and $\lim_{j\to\infty}jx_j^M=1$. In particular,
$(x_j,y_j)\in\widetilde S_k(\varepsilon,\pi/4)\times \mathbb{C}$ for some $k\in\{0,\dots, M-1\}$ and for $j$ large enough.
As in Section \ref{constructing}, if we consider the sectorial coordinates
$(x,z)$ in $\widetilde S_k(\varepsilon,\pi/4)\times \mathbb{C}$, where $z=y-u_k(x)$, 
we can write
$$F(x,z)=\left(x+x^{M+1}\left[a+O_1(x,z)\right], z+x^Mz\left[b+O_1(x,z)\right]\right).$$
Since $x_j^M\to 0$ along $\R^+$ and $\re b>0$, for some $j_0\ge0$ and some $\nu>0$ we have that 
$|z_{j+1}|\ge |z_j|\left(1+\nu |x_j|^M\right)$
for all $j\ge j_0$ and therefore
$$|z_j|\ge |z_{j_0}|\prod_{l=j_0}^{j-1}\left(1+\nu |x_l|^M\right)$$
for all $j\ge j_0+1$. Then, if $z_{j_0}\neq 0$ we have that $|z_j|\to\infty$, which is a contradiction. Hence $z_{j_0}=0$, which means that the orbit of
$(x,y)$ eventually enters the parabolic curve $z=0$ of $F$ in the sectorial domain $\widetilde S_k(\varepsilon,\pi/4)\times \mathbb{C}$. For each $k\in\{0,\dots,M-1\}$, let    
$\mathcal{P}_k$ be the set of points in $\mathcal{U}$ that attracted by $F|_{\mathcal U}$ into the parabolic curve of $F$ in $\widetilde S_k(\varepsilon,\pi/4)\times \mathbb{C}$; it is not difficult to see that $\mathcal P_k$ is a one-dimensional complex submanifold of $\mathcal U$, and we have shown that every point in $\mathcal U$ outside the fixed  set $\{x=0\}$ and outside the manifold $\mathcal P^+ =\mathcal P_0 \cup\dots \cup \mathcal P_{d-1}$ has a finite positive orbit in $\mathcal U$. In the same way, up to reducing $\delta$ and $\varepsilon$, if $\mathcal P_k^-$ is the one-dimensional complex submanifold of $\mathcal U$ of points attracted by $F^{-1}|_{\mathcal U}$ into the parabolic curve of $F^{-1}$ in
the sectorial domain $\widetilde S_k^-(\varepsilon, \pi/4)\times \mathbb{C}$, then every point in $\mathcal U$ outside the fixed set $\{x=0\}$ and outside the manifold $\mathcal P^-=\mathcal P_0^- \cup\dots \cup \mathcal P_{d-1}^-$  has a finite negative orbit in $\mathcal U$. This ends the proof of Theorem~\ref{saddlebehavior}.

\section{Appendix: Resolution theorem for biholomorphisms}
The resolution theorem for two-dimensional biholomorphisms stated in the introduction is valid not only for tangent to the identity biholomorphisms, but more generally for unipotent biholomorphismsm, and is based on the corresponding theorem for vector fields and foliations in $\C^2$. A formal vector field $X$ in $(\mathbb{C}^2,0)$ can be written, in a unique way up to multiplication by a unit, as $X=f(A\partial_x+B\partial_y),$ where $f,A,B\in\C[[x,y]]$ and $A$ and $B$ have no common factor. The vector field $\sat X=A\partial_x+B\partial_y$ is called the saturation of $X$. If $f$ is not a unit, we say that $\sing X=\sqrt{(f)}$ is the singular locus of $X$; if $f$ is a unit, we say that $X$ is saturated and define $\sing X=\{0\}$ if $A$ and $B$ are not units and $\sing X=\emptyset$ otherwise. 
An irreducible formal curve $(g)$ in $(\mathbb{C}^2,0)$  is said to be a separatrix of a formal vector field $X$ if $X(g)\in (g)$. If $X$ is not singular, its formal integral curve through the origin is its only separatrix. The branches of the singular locus of $X$ are separatrices, which are called fixed.  

We say that a saturated singular vector field $X$ in $(\mathbb{C}^2,0)$ is reduced if the eigenvalues $\lambda_{1},\lambda_{2}$ of its linear part satisfy $\lambda_{1}\neq 0$ and $\lambda_{2}/\lambda_{1}\not\in\Q_{>0}$; if $\lambda_2\neq 0$ we say that $X$ is {non-degenerate}, otherwise $X$ is called a {saddle-node}. A reduced vector field $X$ has exactly two formal separatrices, which are non-singular and transverse, and each one is tangent to an eigenspace of the linear part of $X$. Let $X$ be a singular formal vector field in $(\mathbb{C}^2,0)$.
The resolution theorem for vector fields, due to Seidenberg \cite{Sei}, asserts  that there exist a finite composition of blow-ups 
$\pi\colon (M,E)\to(\mathbb{C}^2,0)$, a formal vector field $\tilde{X}$ along $E$ with $\pi_*\tilde{X}=X$ and finitely many points  $p_1,\dots, p_k\in E$ such that $\sat\tilde X_{p_1},\dots,\sat\tilde X_{p_k}$ are reduced and $\sat\tilde X_p$ is not singular for any  $p\in E\backslash\{p_1,\dots,p_k\}$ (throughout this section, if $\mathfrak g$ is any analytic or formal object, we denote by $\mathfrak g_p$ its germ at the point $p$). The set $E$, called the exceptional divisor, is a finite union of smooth rational curves with normal crossings; we say that a point in $E$ is a corner if it is the intersection of two components of $D$. Up to composing $\pi$ with some additional blow-ups, we can assume that the family of separatrices of $X$, even if it is infinite, is desingularized by $\pi$: 
\begin{enumerate}[(a)]
\item\label{aaa}  the strict transform of each separatrix is a  non-singular curve at a non-corner point of $E$ and is transverse to $E$; 
\item\label{bbb} the strict transforms of different separatrices are curves at different points of $E$.
\end{enumerate}
Any map $\pi$ as above is called a {resolution} of $X$ and we also say that $\tilde X$ is a resolution of $X$.  Any further blow-up at a singular point 
of $\tilde X$ in $E$ gives another resolution of $X$. As may be expected, there exists a unique {minimal resolution} of $X$ in the sense that any other resolution is obtained from the minimal one by performing finitely many additional blow-ups. If $X$ is a nilpotent vector field in $(\mathbb{C}^2,0)$ (i.e. its linear part is nilpotent) and $\tilde X$ is a resolution of $X$, then $\tilde X_p$ is nilpotent for any point $p\in E$. In particular, if $X$ is nilpotent then the singular points of $\tilde X$ in $E$ are not isolated: an isolated singularity $p$ of $\tilde X$ would be saturated and reduced, so $\tilde X_p$ would not be nilpotent. Therefore the set $\sing_E\tilde X$ of singular points of $\tilde X$ in $E$ is a union of components of $E$. We define the singular locus $\sing\tilde X$ of $\tilde X$ as the union of $\sing_E\tilde X$ with the strict transform of $\sing X$ by $\pi$. It is easy to see that $(\sing\tilde X)_p=\sing \tilde X_p$ for all $p\in E$. Observe that $E\cup\sing\tilde X$ is the transform of $\sing X$, so it is a finite union of smooth curves with normal crossings, and so is $\sing\tilde X$.
\begin{lemma}\label{redutan}Let $\tilde X$ be a resolution of $X$ and let $p\in E$ such that $\sat\tilde{X}_p$ is singular (hence a reduced singularity). Then each branch of $(E\cup\sing \tilde X)_p$ is one of the two separatrices of $\sat\tilde{X}_p$.
\end{lemma} 
\proof Let $S_1$ and $S_2$ be the separatrices of $\sat\tilde{X}_p$. If both $S_1$ and $S_2$ were not contained in $E$, they would be the strict transforms of different separatrices of $X$ passing through the same point in $E$, contradicting \eqref{aaa}, so we can assume that $S_1$ is contained in a component of $E$. Suppose that there is a branch $S$ of $(E\cup\sing \tilde X)_p$ different from  $S_1$ and $S_2$. If $S$ is a component of $E$ then $p$ is a corner and $S_2$ is the strict transform of a separatrix of $X$ passing through $p$, which contradicts \eqref{aaa}. If $S$ is a branch of $\sing\tilde X_p$ not contained in $E$ then $S$ is the strict transform of a separatrix of $X$, so by \eqref{aaa} $p$ is not a corner and then $S_2$ is also the strict transform of a separatrix of $X$ passing through $p$, contradicting \eqref{bbb}.\qed   

If $\tilde X$ is a resolution of $X$, we classify the components of the exceptional divisor in two types: 
\begin{enumerate}[(1)]
\item A component $D$ of $E$ is invariant if for some point $p\in D$ the germ $D_p$ is a separatrix of $\sat\tilde X_p$. In this case the same happens for any other point in $D$. 
\item\label{gaga} If a component $D$ of $E$ is not invariant, we say that it is dicritical. In this case $D\subset \sing\tilde X$ and, as we will see next, the vector field $\sat\tilde X_p$ is non-singular and transverse to $D$ for all $p\in D$, and any other component of $E$ intersecting $D$ is invariant. 
\end{enumerate}
Let us show the assertions in \eqref{gaga}. By Lemma~\ref{redutan}, if $\sat\tilde X_p$ were singular for some $p\in D$ then $D$ would be invariant, so
$\sat\tilde X_p$ is not singular and its formal integral curve $C$ is different from $D_p$ because $D$ is not invariant. If $p$ is not a corner, then $C$ is  the strict transform of a separatrix of $X$, so $C$ is transverse to $D$; if $p$ is a corner, from \eqref{aaa} we conclude that $C=D'_p$, where $D'$ is the other component of $E$ through $p$ and therefore $\sat\tilde X_p$ is transverse to $D$. This also shows that any component $D'$ of $E$ intersecting $D$ is invariant.  

\strut 

Consider now a unipotent biholomorphism $F$, i.e. $DF(0)=I+N$ where $I$ is the identity and $N$ is nilpotent. In a formal sense, $F$ is the time-$1$ flow of a unique formal vector field in $(\mathbb{C}^2,0)$, denoted $\log F$, which is singular at the origin and has $N$ as linear part. In particular, if $F$ is tangent to the identity then $\log F$ has order at least 2. Moreover, the fixed point set of $F$ coincides with the singular locus of $\log F$, which is therefore convergent. If $\pi$ is the blow-up at the origin, the map $\tilde F=\pi^{-1}\circ F\circ \pi$ is a biholomorphism in a neighborhood of $E=\pi^{-1}(0)$ which leaves $E$ invariant, and satisfies that $\log \tilde F_p=\tilde X_p$ for any fixed point $p\in E$ of $\tilde F$. 

\begin{theorem} Let $F$ be a unipotent biholomorphism in $(\mathbb{C}^2,0)$, let $\pi$ be a resolution of $\log F$ and let $\tilde F=\pi^{-1}\circ F\circ \pi$ be the transform of $F$ by $\pi$. Then, if $p\in E=\pi^{-1}(0)$ is a fixed point of $\tilde F$, the germ $\tilde{F}_p$ is analytically conjugate to one of the reduced models (i), (ii) and (iii) of the introduction. If $F$ is tangent to the identity and $\pi$ is the minimal resolution of $\log F$, then $E$ is pointwise fixed by $\tilde F$. 
\end{theorem}
\proof Let $\tilde X$ be the transform of $X=\log F$ by $\pi$ and let $p\in E$ be a fixed point of $\tilde F$, so $p\in\sing\tilde X$. Suppose that $p$ belongs to a dicritical component $D$ of $E$. Then $D\subset\sing\tilde X_p$ and $\sat \tilde X_p$ is non-singular and transverse to $D$. Since $(D\cup \sing \tilde X)_p$ is smooth or has two smooth transverse branches, we can take holomorphic coordinates $(x,y)$ at $p$ such that $D=\{x=0\}$ and such that
$\{y=0\}$ is the other branch of $\sing \tilde X_p$ if it exists, so $\sing\tilde X_p\subset \{xy=0\}$. Then, up to rescaling the coordinates we have
$\tilde X_p=x^My^N\big[(1+\tilde A(x,y)){\partial}_x+ \tilde B(x,y){\partial}_y\big]$, 
where $\ord\tilde A\ge 1$, $M\ge 1$, $N\ge 0$ and $(M,N)\neq (1,0)$ because
 $\tilde X_p$ is nilpotent. Therefore its time-1 flow $\tilde F_p$ will be of the form (i) 
 once we show that $\tilde B\in (y)$ if $N\ge 1$. Suppose that $N\ge 1$, so 
 $\{y=0\}\subset \sing\tilde X_p$, and let $C$ be the formal integral curve of 
 $(1+\tilde A){\partial_x}+\tilde B {\partial_ y}$ through the origin. If $\{y=0\}$
  is a component of $E$ then $p$ is a corner and, in view of \eqref{aaa}, 
  necessarily $C=\{y=0\}$ and therefore $\tilde B\in (y)$. If $\{y=0\}$ is not a
   component of $E$, then it is the strict transform of a (fixed) separatrix of $X$ so,
    in view of \eqref{bbb}, it has to coincide with $C$ and again $\tilde B\in (y)$. 
    Suppose now that $p$ does not belong to a dicritical component of $E$. We 
    assume first that $\sat\tilde X_p$  is not singular. Take a component $D$ of $E$
     such that $p\in D$. As in the previous case, we have holomorphic coordinates 
     $(x,y)$ at $p$ such that $D=\{y=0\}$ and $\sing\tilde X_p\subset \{xy=0\}$. 
     Since $D$ is invariant, $\{y=0\}$ is the formal integral curve of  $\sat\tilde X_p$
      through $p$, so we can write
       $\sat\tilde X_p=(1+\hat A){\partial_ x}+\hat B{\partial_ y} $ with 
       $\hat B\in (y)$. Then, up to rescaling the coordinates, we have
        $\tilde X_p=x^My^N\big[(1+\tilde A){\partial_ x}+\tilde B{\partial_ y}\big]$, 
        where $M,N\notin\{(0,0),(1,0)\}$ because $\tilde X_p$ is singular and 
        nilpotent, so $\tilde F_p$ is of the form (i). Assume now that $\sat\tilde X_p$ 
        is singular, hence reduced. We suppose first that it is a saddle-node. 
        Let $S$ be the separatrix of $\sat\tilde X_p$ that is tangent to the 
        eigenspace associated to the nonzero eigenvalue of the linear part of
         $\sat\tilde X_p$, and let $S_0$ be the other separatrix. Let $(x,y)$ be
          holomorphic coordinates at $p$ such that $\{x=0\}$ and $\{y=0\}$ 
          are respectively tangent to $S_0$ and $S$, so we can write
           $\sat\tilde X_p=(x+\hat A)\partial_x+\hat B\partial_y$, 
           where $\ord \hat A, \ord \hat B\ge 2$ and $x+\hat A$
            and $\hat B$ have no common factors. We know that $\sing\tilde X_p$
             has one or two branches, which by
 Lemma \ref{redutan} are contained in $\{S,S_0\}$. Thus, we can assume that the coordinates
 are chosen in such a way that if $S_0\subset \sing\tilde X_p$ then $S_0=\{x=0\}$, so
 $\hat A\in (x)$, and if $S\subset \sing\tilde X_p$ then $S=\{y=0\}$, so $\hat B\in (y)$; 
hence $\sing\tilde X_p\subset\{xy=0\}$. Thus, up to rescaling the coordinates,
 we have
 $\tilde X_p=x^My^N\big[(x+\tilde A)\partial_x+\tilde B\partial_y\big]$, where $M+N\ge 1$, 
$\ord \tilde A, \ord\tilde B\ge 2$, $\tilde A\in (x)$ if $M\ge 1$, $\tilde B\in (y)$ if
 $N\ge 1$ and $x+\tilde A$ and $\tilde B$ have no common factors. Hence, 
 $\tilde F_p$ is of the form (iii). Finally, suppose that $\sat\tilde X_p$ is
  non-degenerate. Since $\sing \tilde X_p$ contains at least one branch, in 
  suitable coordinates $(x,y)$ we have $\{x=0\}\subset\sing\tilde X_p$ so, by 
  Lemma \ref{redutan}, the curve $\{x=0\}$ is a separatrix of $\sat\tilde X_p$. 
  As in the previous case, we can assume that the other separatrix $S$ of 
  $\sat\tilde X_p$ is tangent to $\{y=0\}$ and that $S=\{y=0\}\subset\sing\tilde X_p$
   if $\sing\tilde X_p$ has two branches, so $\sing\tilde X_p\subset\{xy=0\}$. 
   Therefore we have 
$\tilde X_p=x^My^N\big[(ax+\tilde Ax){\partial_ x}+(by+\tilde B){\partial_ y}\big]$, where $M\ge 1$, $N\ge 0$, $a,b\in\mathbb{C}^*$, $a/b\notin\mathbb{Q}_{>0}$, $\ord\tilde A\ge 1$, $\ord\tilde B\ge 2$ and $\tilde B\in (y)$ if $N\ge 1$, so $\tilde F_p$ is of the form (ii). 

In order to prove the last assertion of the theorem it suffices to show that if a vector field $X$ has order at least 2 and $\pi$ is its minimal resolution, then $\pi^{-1}(0)\subset\sing\tilde X$. Suppose that this property holds when the minimal resolution of $X$ is achieved with less than $n\in\mathbb{N}$ blow-ups, and let $X$ be a formal vector field with $\ord X\ge2$ whose minimal resolution is obtained with $n$ blow-ups. Let $\sigma$ be the blow-up at the origin and let $\hat X$ be the transform of $X$ by $\sigma$. Since $\ord X\ge 2$, we have that $\hat X$ vanishes on $D=\sigma^{-1}(0)$ with order $\nu\ge 1$ if $D$ is invariant or $\nu\ge 2$ if $D$ is dicritical. So $D$ will be in the singular locus of the resolution of $X$ and, in view of the inductive hypothesis, it is enough to show that $\hat X$ has order at least 2 at each point in $D$ that is blown-up in the resolution. Let $p\in D$ be one of such points. Since $\ord\hat X_p\ge \nu$, it suffices to consider the case where $D$ is invariant and $\nu=1$. We can also assume that $\sing \hat X_p =D_p$ and that $\hat X_p$ vanishes on $D_p$ with multiplicity one, because otherwise $\ord\hat X_p\ge 2$. Then, if $(x,y)$ are holomorphic coordinates at $p$ such that $D_p=\{y=0\}$, we have $\hat X_p=y\big[(a+\tilde A){\partial_ x}+(b+\tilde B){\partial_ y}\big]$, where $\ord A,\ord B\ge 1$, and necessarily $b=0$ because $\hat X_p$ is nilpotent. If $a\neq 0$, we see that $\hat X_p$ is actually in final form, so no further blow-up at $p$ would be necessary. Therefore $a=0$ and $\ord\hat X_p\ge 2$. \qed 

\begin{remark}
The final reduced models (i), (ii) and (iii) we consider correspond to the standard final models for vector fields, and are not exactly the same ones that appear in the resolution theorem in \cite{Aba}, since our set of reduced forms is stable under blow-ups, in the sense that any further blow-up will produce only biholomorphisms in the same set of reduced final models (for example, the dicritical fixed points considered as final models in \cite{Aba} can be reduced by additional blow-ups to non-dicritical models). In our final models, the blow-up of a map of the form (i) produces maps of the form (i), the blow-up of a map of the form (ii) produces maps of the form (i) and (ii) and the blow-up of a map of the form (iii) produces maps of the form (i), (ii) and (iii).
\end{remark}

\end{document}